\documentclass[12pt]{article}

\usepackage{amsmath,amssymb,amsthm, fullpage}

\newtheorem{theorem}{Theorem}[section]

\newtheorem{corollary}[theorem]{Corollary}
\newtheorem{lemma}[theorem]{Lemma}

\newtheorem{proposition}[theorem]{Proposition}

\newtheorem{conjecture}[theorem]{Conjecture}

\numberwithin{equation}{section}

\def\C{{\mathbb C}}
\def\D{{\mathbb D}}

\def\N{{\mathbb N}}

\def\R{{\mathbb R}}

\def\E{{\mathbb E}}
\def\T{{\mathbb T}}
\def\P{{\mathbb P}}

\newcommand{\cE}{\mathcal{E}}
\def\eps{\varepsilon}

\begin{document}

\title{Equidistribution of Zeros of Random Polynomials}

\author{Igor Pritsker and Koushik Ramachandran}

\date{}

\maketitle

\begin{abstract}
We study the asymptotic distribution of zeros for the random polynomials $P_n(z) = \sum_{k=0}^n A_k B_k(z)$, where $\{A_k\}_{k=0}^{\infty}$ are non-trivial i.i.d. complex random variables. Polynomials $\{B_k\}_{k=0}^{\infty}$ are deterministic, and are selected from a standard basis such as Szeg\H{o}, Bergman, or Faber polynomials associated with a Jordan domain $G$ bounded by an analytic curve. We show that the zero counting measures of $P_n$ converge almost surely to the equilibrium measure on the boundary of $G$ if and only if $\E[\log^+|A_0|]<\infty$.
\end{abstract}

\section{Introduction}

Zeros of polynomials of the form $P_n(z)=\sum_{k=0}^{n} A_k z^k,$ where $\{A_n\}_{k=0}^n$ are random coefficients,  have been studied by Bloch and P\'olya, Littlewood and Offord, Erd\H{o}s and Offord, Kac, Rice, Hammersley, Shparo and Shur, Arnold, and many other authors. The early history of the subject with numerous references is summarized in the books by Bharucha-Reid and Sambandham \cite{BR}, and by Farahmand \cite{Fa}. It is well known that, under mild conditions on the probability distribution of the coefficients, the majority of zeros of these polynomials accumulate near the unit circumference, being equidistributed in the angular sense.  Introducing modern terminology, we call a collection of random polynomials $P_n(z)=\sum_{k=0}^n A_k z^k,\ n\in\N,$ the ensemble of \emph{Kac polynomials}.
Let $\{Z_k\}_{k=1}^n$ be the zeros of a polynomial $P_n$ of degree $n$, and define the \emph{zero counting measure}
\[
\tau_n=\frac{1}{n}\sum_{k=1}^n \delta_{Z_k}.
\]
The fact of equidistribution for the zeros of random polynomials can now be expressed via the weak convergence of $\tau_n$ to the normalized arclength measure $\mu_{\T}$ on the unit circumference $\T$, where $d\mu_{\T}(e^{it}):=dt/(2\pi).$ Namely, we have that $\tau_n \stackrel{w}{\rightarrow} \mu_{\T}$ with probability 1 (abbreviated as a.s. or almost surely). More recent work on the global distribution of zeros of Kac polynomials include papers of Ibragimov and Zaporozhets \cite{IZ}, Kabluchko and Zaporozhets \cite{KZ1,KZ2}, etc. In particular, Ibragimov and Zaporozhets \cite{IZ} proved that if the coefficients are independent and identically distributed non-trivial random variables, then the condition $\E[\log^+|A_0|]<\infty$ is necessary and sufficient for $\tau_n \stackrel{w}{\rightarrow} \mu_{\T}$ almost surely. Here, $\E[X]$ denotes the expectation of a random variable $X$, and $X$ is called non-trivial if $\P(X=0)<1$.

Asymptotic distribution of zeros for deterministic polynomials, and especially zeros of sections of a power series, has a long history that dates back to at least the work of Jentzsch and Szeg\H{o}, see e.g. Andrievskii and Blatt \cite{AB} for an overview. It is natural to use the wealth of accumulated results in the study of zeros for random power series. A number of authors followed this approach, and the recent paper of Fern\'andez \cite{Fe} deduces the above result of Ibragimov and Zaporozhets from the criterion of equidistribution of zeros of partial sums due to Carlson-Bourion and Erd\H{o}s-Fried. Fern\'andez shows that the gauge of a random power series satisfying the assumption of Ibragimov and Zaporozhets is equal to 1, so that the Carlson-Bourion characterization of zero equidistribution applies, see \cite{Fe} for details. We present a similar, and perhaps somewhat simpler argument that proves the result of Ibragimov and Zaporozhets. The assumption $\E[\log^+|A_0|]<\infty$ for a sequence of non-trivial i.i.d. complex random variables $\{A_n\}_{n=0}^\infty$ is well known to be equivalent to 
\begin{align*} 
\limsup_{n\to\infty} |A_n|^{1/n} = 1 \quad\mbox{ a.s.},
\end{align*}
while $\E[\log^+|A_0|]=\infty$ holds if and only if the above $\limsup$ is infinite with probability one, see Arnold \cite{Ar}, Ibragimov and Zaporozhets \cite{IZ}, etc. Thus the radius of convergence for the random power series $\sum_{k=0}^\infty A_k z^k$ is either 1 or 0 almost surely. In particular, this series converges in the unit disk a.s. under the assumption $\E[\log^+|A_0|]<\infty$, which immediately gives that no point inside $\D$ can be a point of accumulation for the zeros of partial sums $P_n$ by Hurwitz's theorem. Moreover, the results of Jentzsch and Szeg\H{o} provide us a subsequence of partial sums with equidistributed zeros according to the measure $\mu_\T.$ If the counting measures $\tau_n$ fail to converge to $\mu_\T$ along a subsequence $n_k,\ k\in\N,$ then we show that for some $c,q\in(0,1)$ the coefficients satisfy
\begin{align*} 
|A_n| \le q^n,\quad c n_k\le n\le n_k \quad\mbox{ a.s.}
\end{align*}
The above behavior of the coefficient sequence is often referred to as Hadamard-Ostrowski gaps. On the other hand, we prove that for any sequence of non-trivial i.i.d. random coefficients there is $b>0$ such that 
\begin{align*} 
\liminf_{n\to\infty} \left(\max_{n-b\log{n}<k\le n} |A_k|\right)^{1/n} \ge 1 \quad\mbox{ a.s.}
\end{align*}
Since this is clearly incompatible with Hadamard-Ostrowski gaps, we conclude that the whole sequence of the counting measures $\tau_n$ for the partial sums $P_n$ must converge to $\mu_\T$ with probability one.
If $\E[\log^+|A_0|]=\infty$ then the radius of convergence for our random power series is 0 almost surely. A result of Rosenbloom \cite{Ro} now gives that there is a subsequence of partial sums with all zeros accumulating at the origin. In fact, his Theorem XVIII on pages 40--41 of \cite{Ro} gives more precise information about zeros.

We apply essentially the same approach to prove a generalization of zero equidisribution criterion for sums of random power series spanned by various bases, e.g., by orthogonal polynomials. Sufficient conditions for almost sure equidistribution of zeros of random orthogonal polynomials were considered by Shiffman and Zelditch \cite{SZ1} and \cite{SZ2}, Bloom \cite{Bl1} and \cite{Bl2}, Bloom and Shiffman \cite{BS}, Bloom and Levenberg \cite{BL}, Bayraktar \cite{Ba1} and \cite{Ba2}, and others. Pritsker \cite{Pr2} and \cite{Pr3} considered zero distribution for random polynomials spanned by general bases.

\section{Equidistribution of Zeros for Random Sums of Polynomials}

We now consider more general ensembles of random polynomials 
$$P_n(z) = \sum_{k=0}^n A_k B_k(z)$$
spanned by various bases $\{B_k\}_{k=0}^{\infty}.$ Let $B_k(z)=\sum_{j=0}^k b_{j,k}z^j$, where $b_{j,k}\in\C$ for all $j$ and $k$, and $b_{k,k}\neq 0$ for all $k$, be a polynomial basis, so that $\text{deg}\  B_k=k$ for all $k\in\N\cup\{0\}.$ Given a compact set $E\subset\C$ of positive logarithmic capacity $\text{cap}(E)$, we denote the equilibrium measure of $E$ by $\mu_E$, which is a positive unit Borel measure supported on the outer boundary of $E$, see \cite{Ra} for background. It is known that under rather weak assumptions on the random coefficients $A_k$ and the basis polynomials $B_k$ associated with the set $E$, the zeros of random polynomials $P_n$ are almost surely equidistributed according to the measure $\mu_E$. This means that the counting measures $\tau_n$ in zeros of $P_n$ converge weakly to $\mu_E$ with probability one.

If $E$ is a finite union of rectifiable curves and arcs, we call the polynomials orthonormal with respect to the arclength measure $ds$ by Szeg\H{o} polynomials. When $E$ is a compact set of positive area, we call the polynomials orthonormal with respect to the area measure $dA$ on $E$ by Bergman polynomials. The basis of Faber polynomials is defined for any compact set $E$ with simply connected unbounded component $\Omega$ of the complement $\overline{\C}\setminus E.$ The $n$-th Faber polynomial is the polynomial part of Laurent expansion for $\Phi^n(z)$ at $\infty$, where $\Phi:\Omega\to\Delta$ is the canonical conformal mapping of $\Omega$ onto $\Delta:=\overline{\C}\setminus\overline{\D}$ normalized by $\Phi(\infty)=\infty$ and $\Phi'(\infty)>0.$

\begin{theorem} \label{thm2.1}
Suppose that $E$ is the closure of a Jordan domain $G$ with analytic boundary $L,$ and that the basis $\{B_k\}_{k=0}^{\infty}$ is given either by Szeg\H{o}, or by Bergman, or by Faber polynomials. Assume further that the random coefficients $\{A_k\}_{k=0}^{\infty}$ are non-trivial i.i.d. complex random variables. The zero counting measures of $P_n(z) = \sum_{k=0}^n A_k B_k(z)$ converge almost surely to $\mu_E$ if and only if $\E[\log^+|A_0|]<\infty$.
\end{theorem}

Our proofs show that the above result holds true for many other standard bases used to expand analytic functions, e.g. for Lagrange interpolation polynomial basis as well as for various extremal polynomials. In the case of the unit disk and the Faber basis, our result reduces to that of Ibragimov and Zaporozhets.
 
In the proof of Theorem \ref{thm2.1}, we obtain additional useful facts summarized below.  
 
\begin{corollary} \label{cor2.2}
Suppose that $E$ is the closure of a Jordan domain $G$ with analytic boundary $L,$ and that the basis $\{B_k\}_{k=0}^{\infty}$ is given either by Szeg\H{o}, or by Bergman, or by Faber polynomials. If $\{A_k\}_{k=0}^{\infty}$ are non-trivial i.i.d. complex random variables such that $\E[\log^+|A_0|]<\infty$, then the random polynomials $P_n(z) = \sum_{k=0}^n A_k B_k(z)$ converge almost surely to a random analytic function $f$ that is not identically zero. Moreover,
\begin{align} \label{2.1}
\lim_{n\to\infty} |P_n(z)|^{1/n} = |\Phi(z)|, \quad z\in\Omega,
\end{align}
holds with probability one.
\end{corollary}
Note that \eqref{2.1} implies divergence of $P_n$ in $\Omega$ with probability one, so that $L=\partial G$ is the natural boundary for the random series $\sum_{k=0}^\infty A_k B_k(z)$. This phenomenon is well known in the case of standard power series, see \cite{Ka}, for which the circle of convergence is almost surely the natural boundary. More details and general results related to \eqref{2.1} may be found in \cite{BL}.

We know from the Ibragimov-Zaporozhets theorem and Theorem \ref{thm2.1} that if $\E[\log^+|A_0|] = \infty,$ then the equidistribution of zeros does not hold with probability one. Behavior of zeros for random power series in this case was studied in \cite{GZ} and \cite{KZ1}. However, it is possible to find a polynomial basis and i.i.d. coefficients for which there is still a subsequence of $\tau_n$ that converges to $\mu_{\mathbb{T}}$ a.s. We give such a construction below. Recall that a positive measurable function $L$ is called slowly varying (at infinity) if it is defined in a neighborhood of infinity, and if for every $t>0$ we have
$$\lim_{x\to\infty}\dfrac{L(tx)}{L(x)} = 1.$$
A random variable $X$ is said to have a slowly varying tail if $L(x) = \mathbb{P}\left(|X| > x\right)$ is a slowly varying function. For example, if  $\mathbb{P}\left(|X| > x\right) = h(x)/(\log{x})^\alpha$ for all large $x\in\R$, where $\alpha>0$ and $\lim_{x\to\infty} h(x) = c > 0$, then $X$ has a slowly varying tail. Note that in this example $\E[\log^+|X|]=\infty$ if and only if $\alpha\le 1.$

\begin{proposition} \label{prop2.3}
Let $\{A_k\}_{k=0}^\infty$ be a sequence of i.i.d. random variables with $A_0\geq 1$ almost surely. Consider the sequence of random polynomials $P_n(z) = \sum_{k=0}^{n}A_k(z^k-1),\ n\in\N.$\\ 
\textup{(i)} If $\E[\log A_0]<\infty$ then $\tau_n \stackrel{w}{\rightarrow} \mu_{\mathbb{T}}$ as $n\to\infty$ almost surely.\\
\textup{(ii)} If $\E[\log A_0]=\infty$ and $A_0$ has a slowly varying tail,  then there is a subsequence such that 
$\tau_{n_k} \stackrel{w}{\rightarrow} \mu_{\mathbb{T}}$ a.s.
\end{proposition}

\medskip
If the random coefficients satisfy mild assumptions such as in Theorem \ref{thm2.1}, then the zero counting measures of random polynomials $P_n(z) = \sum_{k=0}^n A_k B_k(z)$ converge almost surely to $\mu_E$ for very general sets $E$ and associated bases $\{B_k\}_{k=0}^\infty.$ We direct the reader to the recent papers \cite{BL}, \cite{Ba1,Ba2} and \cite{Pr2,Pr3}, and to references found therein. However, the necessity part of Theorem \ref{thm2.1} seems to be open in such general setting. We end this section with conjectures related to this matter. We assume here that $G$ is an arbitrary Jordan domain, and $E=\overline{G}.$

\begin{conjecture} \label{con2.4} It is possible to construct sets $E\subset\C$ with non-analytic boundary, and sequences of i.i.d. random variables $\{A_k\}_{k=0}^\infty$ with $\E[\log^+|A_0|]=\infty$, such that for each 
basis of Szeg\H{o}, Bergman, or Faber polynomials there is a subsequence of zero counting measures for $P_n(z) = \sum_{k=0}^n A_k B_k(z)$ satisfying $\tau_{n_k} \stackrel{w}{\rightarrow} \mu_E$ a.s. 
\end{conjecture}

Our next conjecture is motivated by Proposition \ref{prop2.3}.

\begin{conjecture} \label{con2.5} It is possible to construct a weight function $w\ge 0$ on $\T$, and a sequence of i.i.d. random variables $\{A_k\}_{k=0}^\infty$ with $\E[\log^+|A_0|]=\infty$, such that for the basis of orthonormal on $\T$ polynomials with respect to this weight, we have a subsequence of zero counting measures for $P_n(z) = \sum_{k=0}^n A_k B_k(z)$ satisfying $\tau_{n_k} \stackrel{w}{\rightarrow} \mu_{\mathbb{T}}$ a.s.

An analogous construction should also exist for area orthonormal polynomials with respect to an appropriate weight function $w\ge 0$ on $\D.$
\end{conjecture}

\section{Proofs} \label{secP}

We start with auxiliary results on random coefficients. The first lemma is well known, but we prove it for the convenience of the reader. 

\begin{lemma}\label{lem3.1}
If $\{A_k\}_{k=0}^{\infty}$ are non-trivial, independent and identically distributed complex random variables that satisfy $\E[\log^+|A_0|]<\infty$, then
\begin{align} \label{3.1}
\limsup_{n\to\infty} |A_n|^{1/n}= 1 \quad\mbox{ a.s.},
\end{align}
and
\begin{align} \label{3.2}
\limsup_{n\to\infty} \left(\max_{0\le k\le n} |A_k| \right)^{1/n} = 1 \quad\mbox{ a.s.}
\end{align}
\end{lemma}

This follows from the Borel-Cantelli Lemmas stated below (see, e.g., \cite[p. 96]{Gut}) in a standard way.

\medskip
\noindent\textbf{Borel-Cantelli Lemmas.} \textit{Let $\{\cE_n\}_{n=1}^{\infty}$ be a sequence of arbitrary events.}\\
(i) \textit{If $\sum_{n=1}^{\infty} \P(\cE_n) < \infty$ then $\P(\cE_n \mbox{ occurs infinitely often})=0.$}\\
(ii) \textit{If events $\{\cE_n\}_{n=1}^{\infty}$ are independent and $\sum_{n=1}^{\infty} \P(\cE_n) = \infty$, then $\P(\cE_n \textup{ i.o.})=1.$}

\begin{proof}[Proof of Lemma \ref{lem3.1}]
For any fixed $\eps>0$, define events $\cE_n=\{|A_n| \ge e^{\eps n}\}, \ n\in\N.$ Then
\begin{align*}
\sum_{n=1}^{\infty} \P(\cE_n) &= \sum_{n=1}^{\infty} \P(\{\log^+|A_n| \ge \eps n\}) = \sum_{n=1}^{\infty} \P\left(\left\{\frac{1}{\eps}\log^+|A_0| \ge n\right\}\right) \\ &\le \frac{1}{\eps} \E[\log^+|A_0|]<\infty.
\end{align*}
Hence $\P(\cE_n \mbox{ occurs infinitely often})=0$ by the first Borel-Cantelli Lemma, so that the complementary event $\cE_n^c$ must happen for all large $n$ with probability 1. This means that $|A_n|^{1/n} \le e^{\eps}$ for all sufficiently large $n\in\N$ almost surely.  We obtain that
\[
\limsup_{n\to\infty} |A_n|^{1/n} \le e^{\eps} \quad\mbox{a.s.},
\]
and since $\eps>0$ was arbitary, this shows that 
\begin{equation}\label{limsup}
\limsup_{n\to\infty} |A_n|^{1/n}\leq 1 \quad\mbox{ a.s.}
\end{equation} 

On the other hand, $A_0$ is non-trivial, so there exists $c>0$ such that $\P(|A_0| > c) >0.$ Therefore since $A_n$ are i.i.d,
\begin{equation}
\sum_{n=1}^\infty \mathbb{P}(|A_n|> c) = \infty.
\end{equation}
Using the second Borel-Cantelli Lemma, it follows that with probability one, there exist infinitely many $n$ such that $|A_n| > c.$ Combining this with  \eqref{limsup}, we obtain \eqref{3.1}. An elementary argument shows that \eqref{3.2} is a consequence of \eqref{3.1}.
\end{proof}

\begin{lemma}\label{lem3.2}
If $\{A_k\}_{k=0}^{\infty}$ are non-trivial i.i.d. complex random variables, then
there is $b>0$ such that
\begin{align} \label{3.3}
\liminf_{n\to\infty} \left(\max_{n-b\log{n}<k\le n} |A_k|\right)^{1/n} \ge 1 \quad\mbox{ a.s.}
\end{align}
\end{lemma}

\begin{proof}
We use a modified idea of Fern\'andez \cite{Fe} in this proof. Let $\alpha_n\le n,\ n\in\N,$ be a sequence of natural numbers that will be specified later. Consider 
\begin{align*}
M_n:=\max_{n-\alpha_n<k\le n} |A_k|.
\end{align*}
The statement 
\begin{align*}
\liminf_{n\to\infty} \left(M_n\right)^{1/n} \ge 1 \quad\mbox{ a.s.}
\end{align*}
is equivalent to 
\begin{align*}
\P(\{M_n \le \lambda^n \mbox{ i.o.}\}) = 0
\end{align*}
for all positive $\lambda<1.$ The latter would follow from the first Borel-Cantelli Lemma if we show that 
\begin{align*}
\sum_{n=1}^\infty \P(\{M_n \le \lambda^n\}) < \infty
\end{align*}
for all positive $\lambda<1.$
Since our variables $\{A_k\}_{k=0}^{\infty}$ are i.i.d., we have
\begin{align*}
\P(\{M_n \le \lambda^n\}) = \P(\{|A_0| \le \lambda^n\})^{\alpha_n}.
\end{align*}
As $A_0$ is non-trivial, we can find $c>0$ and $p\in(0,1)$ such that $\P(\{|A_0| \le c\})\le p$. Hence for any $\lambda<1$ there is $N=N(p,\lambda)\in\N$ such that $\P(\{|A_0| \le \lambda^n\})\le p$ for all $n\ge N.$ This gives 
\begin{align*}
\sum_{n=1}^\infty \P(\{M_n \le \lambda^n\}) \le \sum_{n=1}^\infty p^{\alpha_n} < \infty,
\end{align*}
provided $p^{\alpha_n} \le 1/n^2$ for large $n$. It suffices to take $\alpha_n \ge (-2/\log{p})\log{n}$ to satisfy the latter condition. 
\end{proof}

\medskip
We use the following result of Grothmann \cite{Gr} on the distribution of zeros of polynomials. More details and applications of this result may be found in \cite{AB}.
Let $E\subset\C$ be a compact set of positive capacity such that $\Omega=\overline{\C} \setminus E$ is connected and regular. The Green function of $\Omega$ with pole at $\infty$ is denoted by $g_\Omega(z,\infty)$. We use $\|\cdot\|_K$ for the supremum norm on a compact set $K$.

\noindent
{\bf Theorem G.} {\em If a sequence of polynomials $P_n(z),\ \deg(P_n)\le n\in\N,$ satisfies
\begin{align} \label{3.4}
\limsup_{n\to\infty} \|P_n\|_E^{1/n} \le 1,
\end{align}
for any closed set $K\subset E^\circ$
\begin{align} \label{3.5}
\lim_{n\to\infty} \tau_n(K) = 0,
\end{align}
and there is a compact set $S\subset\Omega$ such that 
\begin{align} \label{3.6}
\liminf_{n\to\infty} \max_{z\in S} \left(\frac{1}{n} \log|P_n(z)| - g_\Omega(z,\infty) \right) \ge 0,
\end{align}
then the zero counting measures $\tau_n$ of $P_n$ converge weakly to $\mu_E$ as $n\to\infty.$}

We first give a short proof of the Ibragimov-Zaporozhets equidistribution criterion for the zeros of partial sums $P_n(z)=\sum_{k=0}^{n} A_k z^k$ of a random power series with non-trivial i.i.d. complex random coefficients $\{A_n\}_{n=0}^\infty$. If $\E[\log^+|A_0|]<\infty$ then the radius of convergence of our random series is almost surely 1 by Lemma \ref{lem3.1}. On this event, the series converges almost surely in the unit disk to a not identically zero analytic function, so that \eqref{3.5} holds for any compact $K\subset\D$ by Hurwitz's theorem. Moreover, \eqref{3.4} is also satisfied almost surely for $E=\overline{\D}$ by \eqref{3.2} and the estimate $\|P_n\|_\T \le (n+1) \max_{0\le k\le n} |A_k|.$ If $\tau_n$ do not converge to $\mu_\T$, then \eqref{3.6} does not hold with probability one for any compact set  $S$ in $\Delta=\overline{\C}\setminus\overline{\D}.$ In particular, for $S=\{|z|=R,\ R>1\}$ and a subsequence $n_m,\ m\in\N,$ we have with positive probability that
\[
\limsup_{m\to\infty} \|P_{n_m}\|_S^{1/n_m} < R,
\]
as $g_\Delta(z,\infty)=\log|z|.$ Hence, with positive probability,
\[
|A_n| \le \frac{1}{2\pi} \int_{|z|=R} \frac{|P_{n_m}(z)|\,|dz|}{R^{n+1}} \le q^n, \quad c n_m \le n\le n_m.
\]
for some $c,q\in(0,1)$. Existence of such Hadamard-Ostrowski gaps contradicts \eqref{3.3} of Lemma \ref{lem3.2}. We conclude that $\tau_n$ must converge to $\mu_\T$ with probability one. If $\E[\log^+|A_0|]=\infty$ then the radius of convergence of the random power series is 0 almost surely. On this event, Theorem XVIII of Rosenbloom \cite[pp. 40--41]{Ro} gives a subsequence of $P_n$ whose zeros are equidistributed near the circles $|z|=|A_n|^{-1/n}.$ In particular, all zeros of this subsequence tend to the origin as $n\to\infty.$

\begin{proof}[Proof of Theorem \ref{thm2.1} and of Corollary \ref{cor2.2}]
We have that $E$ is the closure of a Jordan domain $G$ bounded by an analytic curve $L$ with exterior $\Omega.$ It is well known that the conformal mapping $\Phi:\Omega\to\Delta,\ \Phi(\infty)=\infty,\ \Phi'(\infty)>0,$ extends through $L$ into $G$, so that $\Phi$ maps a domain $\Omega_r$ containing $\overline{\Omega}$ conformally onto $\{|z|>r\}$ for some $r\in(0,1).$ In particular, the level curves of $\Phi$ denoted by $L_\rho$ are contained in $G$ for all $\rho\in(r,1)$, $L_1=L$ and $L_\rho\subset\Omega$ for $\rho>1.$ Since $g_\Omega(z,\infty)=\log|\Phi(z)|$, $L_\rho$ are also the level curves of the Green function. It is also known that in all three cases of polynomial bases we consider in this theorem, we have that 
\begin{align} \label{3.7}
\lim_{n\to\infty} |B_n(z)|^{1/n} = |\Phi(z)|
\end{align}
holds uniformly on compact subsets of $\Omega_r.$ Hence for any compact set $K\subset G$, we have (cf. \cite[pp. 290 and 338]{SL} and \cite[Section 2.3]{Su}) that 
\[
\limsup_{n\to\infty} \|B_n\|_K^{1/n} < 1.
\] 
If the assumption $\E[\log^+|A_0|]<\infty$ is satisfied, then the random series must converge on compact subsets of $G$ almost surely by \eqref{3.1} of Lemma \ref{lem3.1}. Furthermore, its limit is (almost surely) an analytic function that cannot vanish identically because of \eqref{3.1} and uniqueness of series expansions in Szeg\H{o}, Bergman and Faber polynomials (see \cite[pp. 293 and 340]{SL} and Section 6.3 of \cite{Su} for these facts). This proves the corresponding part of Corollary \ref{cor2.2}. We also conclude that \eqref{3.5} holds for any compact $K\subset G$ by Hurwitz's theorem. Note that \eqref{3.4} is also satisfied for $E$ almost surely by \eqref{3.2}, \eqref{3.7} and the estimate 
\[
\|P_n\|_E \le \sum_{k=0}^n |A_k| \|B_k\|_E \le (n+1) \max_{0\le k\le n} |A_k| \, \max_{0\le k\le n} \|B_k\|_E,
\]
as
\[
\limsup_{n\to\infty} \left(\max_{0\le k\le n} \|B_k\|_E\right)^{1/n} \le 1.
\]
If $\tau_n$ do not converge to $\mu_E$ a.s., then \eqref{3.6} cannot hold a.s. for any compact set  $S$ in $\Omega.$ We choose $S=L_R,$ with $R>1$, and find a subsequence $n_m,\ m\in\N,$ such that
\begin{align} \label{3.8}
\limsup_{m\to\infty} \|P_{n_m}\|_{L_R}^{1/n_m} < R,
\end{align}
holds with positive probability. Note that all zeros of $B_n$ are contained outside $\Omega_r$ and hence inside $L_R$ for all large $n$ by \eqref{3.7}. This allows us to write an integral representation
\begin{align} \label{3.9}
A_n = \frac{1}{2\pi i} \int_{L_R} \frac{P_n(z)\,dz}{z B_n(z)},
\end{align}
which is valid for all large $n\in\N$ because $P_n(z)/(z B_n(z)) = A_n/z + O(1/z^2)$ for $z\to\infty.$ We now need more precise estimates for $B_n$ that follow from known asymptotic relations in $\Omega_r$. In the case of Bergman polynomials, Carleman's asymptotic \cite[p. 12]{Ga} implies that there are positive constants $c_1$ and $c_2$ that do not depend on $n$ and $z$, such that
\begin{align} \label{3.10}
c_2 \sqrt{n}\, \rho^n \le |B_n(z)| \le c_1 \sqrt{n}\, \rho^n, \quad z\in L_{\rho},\ \rho>r,\ n\in\N.
\end{align}
Similar estimates, but without the factor $\sqrt{n}$ on both sides are true for Szeg\H{o} and Faber polynomials (see \cite{SL} and \cite{Su}), so that the following proof remains essentially the same for those bases. We estimate from \eqref{3.9} and \eqref{3.10} with $\rho=R$ that
\[
|A_n| \le \frac{|L_R|}{2\pi d} \frac{\|P_n\|_{L_R}}{c_2 \sqrt{n}\, R^n},
\]
where $|L_R|$ is the length of $L_R$ and $d:=\min_{z\in L_R} |z|.$ It follows that
\[
\|P_{n-1}\|_{L_R} \le \|P_n\|_{L_R} + |A_n| \|B_n\|_{L_R} \le \|P_n\|_{L_R} \left(1 + \frac{|L_R|}{2\pi d} \frac{c_1}{c_2} \right) =: C\, \|P_n\|_{L_R}, \quad n\in\N.
\]
Applying this estimate repeatedly, we obtain that
\[
\|P_{n-k}\|_{L_R} \le  C^k\, \|P_n\|_{L_R}, \quad k\le n,
\]
so that \eqref{3.9} yields 
\[
|A_{n-k}| \le \frac{|L_R|}{2\pi d} \frac{\|P_{n-k}\|_{L_R}}{c_2 \sqrt{n-k}\, R^{n-k}} \le \frac{|L_R|}{2\pi d} \frac{C^k\, \|P_n\|_{L_R}}{c_2 \sqrt{n-k}\, R^{n-k}}.
\]
Choosing sufficiently small $\eps>0$ and using \eqref{3.8}, we deduce from previous inequality that
\[
|A_{n_m-k}| \le q^{n_m}, \quad 0\le k\le \eps n_m,
\]
for some $q\in(0,1)$ and all sufficiently large $n_m$, with positive probability. The latter estimate clearly contradicts \eqref{3.3} of Lemma \ref{lem3.2}. Hence \eqref{3.6} holds for $S=L_R,$ with any $R>1$, and $\tau_n$ converge weakly to $\mu_E$ with probability one. Note that \eqref{3.6} for $S=L_R,$ with $R>1$, is equivalent to \eqref{2.1}. Indeed, we have equality in \eqref{3.6}, with $\lim$ instead of $\liminf$, by Bernstein-Walsh inequality and \eqref{3.4}, see Remark 1.2 of \cite[p. 51]{AB} for more details. This concludes the proof of sufficiency part for Theorem \ref{thm2.1} as well as the proof of Corollary \ref{cor2.2}.

Now we prove the necessity part of Theorem \ref{thm2.1}. Let us assume that $\E[\log^+|A_0|] =\infty.$ Recall the bounds of \eqref{3.10} for Bergman polynomials, and recall that similar bounds without the factor $\sqrt{n}$ are true for the Szeg\H{o} and the Faber polynomial bases. It follows that for $n$ large enough, $B_n$ has all its zeros lying inside $L_{\rho}.$ We show below that for infinitely many $n,$ the zeros of $P_n$ all lie inside $L_{\rho}$ where $\rho\in (r, 1).$ Hence the counting measures $\tau_n$ do not converge to $\mu_E.$
  
For the proof, we follow a technique from \cite{IZ}. Let us fix $\rho,$ where $r < \rho < 1.$ To begin with, we note that since $A_k$ are i.i.d. random variables satisfying $\E[\log^+|A_0|]= \infty,$ an application of the Borel-Cantelli lemma gives us that 
$$\limsup_{k\to\infty} |A_k|^{1/k} = \infty \quad\mbox{a.s.}\hspace{0.05in}$$
This means that with probability one, for infinitely many values of $n,$ all three of the following estimates hold:

\begin{equation}\label{E1}
|A_n|^{1/n}\geq \max_{1\leq k\leq n-1}|A_k|^{1/k},
\end{equation}

\begin{equation}\label{E2}
 |A_n|^{1/n}\geq\dfrac{c_2 + 2c_1}{c_2\rho},
\end{equation}

\begin{equation}\label{E3}
|A_0| < \frac{\sqrt{n}}{2}\dfrac{c_2}{c_1}\left(\dfrac{c_2 + 2c_1}{c_2}\right)^n.
\end{equation}

\noindent Fix a large $n$ for which \eqref{E1}, \eqref{E2} and \eqref{E3} hold. Let $z$ be in $L_{\rho}$ and $1\leq k\leq n-1.$ Using inequality \eqref{3.10} and the bound \eqref{E1}, we estimate

\begin{eqnarray}\label{basic}
|A_kB_k(z)|& \le & |A_k|c_1\sqrt{k}\rho^k \nonumber\\
          & \le & c_1\sqrt{n}|A_n|^{k/n}\rho^k \nonumber\\
          & = & c_1\sqrt{n} |A_n\rho^n|^{k/n}.
\end{eqnarray}
Furthermore, \eqref{basic} and \eqref{E3} imply that with probability one,

\begin{align*}\label{chain}
&|A_0B_0 + A_1B_1(z) +...+ A_{n-1}B_{n-1}(z)| \leq |A_0B_0| + ...+ |A_{n-1}B_{n-1}(z)| \\
&\leq \frac{\sqrt{n}}{2} c_2 \left(\frac{c_2 + 2c_1}{c_2}\right)^n + c_1\sqrt{n}|A_n\rho^n|^{\frac{1}{n}} + c_1\sqrt{n}|A_n\rho^n|^{\frac{2}{n}}+...+ c_1\sqrt{n}|A_n\rho^n|^{\frac{n-1}{n}} \\
&\leq \frac{\sqrt{n}}{2} c_2 \left(\frac{c_2 + 2c_1}{c_2}\right)^n - c_1\sqrt{n} + c_1\sqrt{n}\left(\dfrac{|A_n\rho^n| - 1}{|A_n\rho^n|^{\frac{1}{n}}- 1}\right) \\
&\leq \frac{\sqrt{n}}{2} c_2|A_n\rho^n| + c_1\sqrt{n}\left(\dfrac{|A_n\rho^n| - 1}{\frac{c_2+2c_1}{c_2\rho}\rho - 1}\right) - c_1\sqrt{n}\\
&< \frac{\sqrt{n}}{2} c_2|A_n\rho^n| + \frac{c_2\sqrt{n}}{2}(|A_n\rho^n| - 1) \\
&< c_2\sqrt{n} |A_n\rho^n|\leq |A_nB_n(z)|.
\end{align*} 
We remark that in the above chain of inequalities, we used \eqref{E2} to go from the the fourth line to the fifth line and \eqref{3.10} in the last line. Having established the estimate $|A_0B_0 + A_1B_1(z) +...+ A_{n-1}B_{n-1}(z)|< |A_nB_n(z)|$ on $L_{\rho}$, Rouche's theorem now implies that $P_n$ has the same number of zeros as $B_n$ inside $L_{\rho}.$ In other words, all the zeros of $P_n$ are inside $L_{\rho}.$ Since this is true for infinitely many $n,$ and remembering that $\mu_E$ is supported on $L_1 = \partial E,$ we conclude that with probability one, the measures $\tau_n$ do not converge to $\mu_E.$ 

\end{proof}

\begin{proof}[Proof of Proposition \ref{prop2.3}]
\textup{(i)} Since $\E[\log^+|A_0|]< \infty,$ Lemma \ref{3.1} gives that $\limsup_{n\to\infty} |A_n|^{1/n} = 1$ a.s. Coupled with our assumption that $A_n\geq 1$ this gives

\begin{equation}\label{limit}
\lim_{n\to\infty} |A_n|^{1/n} = 1 \hspace{0.1in}\mbox{ a.s.}\hspace{0.1in}
\end{equation}
Next, note that  $P_n(z) = A_nz^n + A_{n-1}z^{n-1}+...+ A_1z -(A_1 + A_2 +... + A_n)$. Dividing out by $A_n$, we have that the zeros of $P_n$ coincide with those of $$Q_n(z) = z^n + \dfrac{A_{n-1}}{A_n}z^{n-1} + ... + \dfrac{A_1}{A_n}z - \dfrac{A_1 + A_2 +... + A_n}{A_n}.$$ 
We now use the following theorem from \cite{BSS} to study the limiting behavior of the zeros of $Q_n.$

\vspace{0.1in}

\noindent 

\textbf{Theorem BSS.} {\em Let $E\subset\C$ be a compact set, $\textup{cap}(E)>0$.  If a sequence of polynomials $P_n(z) = \sum_{k=0}^n c_{k,n} z^k$ satisfy
\begin{align} \label{4.7}
\limsup_{n\to\infty} \|P_n\|_E^{1/n} \le 1 \quad \mbox{and} \quad \lim_{n \rightarrow \infty} |c_{n,n}|^{1/n}=1/\textup{cap}(E),
\end{align}
and for any closed set $A$ in the bounded components of $\C\setminus \textup{supp}\,\mu_E$ we have
\begin{align} \label{4.8}
\lim_{n\to\infty} \tau_n(A) = 0,
\end{align}
then the zero counting measures $\tau_n$ converge weakly to $\mu_E$ as $n\to\infty.$}

It is known that \eqref{4.8} holds if every bounded component of $\C\setminus \textup{supp}\,\mu_E$ contains a compact set $K$ such that 
\begin{align} \label{4.9}
\liminf_{n\to\infty} \|P_n\|_K^{1/n} \ge 1,
\end{align}
see Bloom \cite[p. 1706]{Bl1}.

Continuing with our proof, we observe that $$\|Q_n\|_{\mathbb {T}} \leq\dfrac{2 (A_1 + A_2 + .. + A_n)}{A_n}\leq\dfrac{2n\max_{1\leq i\leq n}{A_i}}{A_n}.$$ 
Using \eqref{3.2} and \eqref{limit}, we obtain $\limsup_{n\to\infty}\|Q_n\|_{\mathbb {T}}^{1/n}\le 1$ a.s. Since $Q_n$ is monic, the condition on the leading coefficient is trivially satisfied and hence \eqref{4.7} holds a.s. in our case. All that remains to check is \eqref{4.8}, which in turn will follow from \eqref{4.9}.  By taking $K =\{0\},$ and using that our random variables are positive, we have 
$$|Q_n(0)|^{1/n} = \left(\dfrac{A_1 + A_2 + .. + A_n}{A_n}\right)^{1/n}\geq 1\hspace{0.05in}\mbox{a.s.}\hspace{0.05in}$$ 
Hence \eqref{4.9} holds a.s., and Theorem BSS now yields that $\tau_n \stackrel{w}{\rightarrow}\mu_{\mathbb{T}}$ a.s.

\textup{(ii)} Let $M_n = \max_{1\leq i\leq n} A_i,$ then since $A_0$ has a slowly varying tail, c.f. \cite{D},
\begin{equation}\label{slowv}
\dfrac{A_1 + A_2 +...+ A_n}{M_n}\hspace{0.1in}\xrightarrow[]{P}\hspace{0.05in} 1.
\end{equation}
We also have that $\E[\log^+|A_0|] = \infty,$ which implies that $\limsup_{k\to\infty} A_k^{1/k} = \infty$ a.s. As before, it follows that for some subsequence $\mathcal{N},$
\begin{equation}
A_n^{1/n}\geq \max_{1\leq k\leq n-1}A_k^{1/k},\quad n\in\mathcal{N}.
\end{equation}
Since $A_k\geq 1$ a.s., this means that along this subsequence $A_n = M_n.$ Now using \eqref{slowv}, and choosing a further subsequence which we continue to call $\mathcal{N},$ we have that along the subsequence $\mathcal{N}$ 
\begin{equation}\label{converge}
\dfrac{A_1 + A_2 +...+ A_n}{A_n}\to 1\hspace{0.05in}\hspace{0.1in}\mbox{a.s.}
\end{equation}
As in the proof of part (i), it is enough to study the zeros of 
$$Q_n(z) = z^n + \dfrac{A_{n-1}}{A_n}z^{n-1} + ... + \dfrac{A_1}{A_n}z - \dfrac{A_1 + A_2 +... + A_n}{A_n}.$$ 
We once again use Theorem BSS to study the limiting behavior of the zeros of $Q_n,$ $n\in\mathcal{N}.$ To begin with, we have 
$$\|Q_n\|_{\mathbb {T}} \leq\dfrac{2 (A_1 + A_2 + .. + A_n)}{A_n}.$$ 
Now using \eqref{converge}, we see that along $\mathcal{N},$ $\limsup_{n\to\infty}\|Q_n\|_{\mathbb {T}}^{1/n}\le 1.$ Since $Q_n$ is monic, the condition on the leading coefficient is trivially satisfied. We show that \eqref{4.9} holds for $K =\{0\}$. Indeed, we have 
$$|Q_n(0)|^{1/n} = \left(\dfrac{A_1 + A_2 + .. + A_n}{A_n}\right)^{1/n},$$ 
so that \eqref{4.9} holds a.s. by \eqref{converge}. Therefore along the subsequence $\mathcal N$ we have $\tau_n \stackrel{w}{\rightarrow}\mu_{\mathbb{T}}$ a.s.

\end{proof}

\textbf{Acknowledgement.} Research of the first author was partially supported by the National Security Agency (grant H98230-15-1-0229) and by the American Institute of Mathematics.

{\emph 
Department of Mathematics,

Oklahoma State University

Stillwater, OK 74074

Email : igor@math.okstate.edu

        koushik.ramachandran@okstate.edu
}

\end{document}